\documentclass[12pt, reqno]{amsart}

\usepackage{amsfonts}
\usepackage{latexsym}
\usepackage{amssymb}
\usepackage{amsmath}
\usepackage{dsfont}
\usepackage{multicol} 
\usepackage[pdftex,svgnames]{xcolor}
\usepackage{url}
\usepackage[utf8]{inputenc}
\usepackage{mathtools}
\usepackage{enumerate}
\usepackage[bookmarks=true, hyperindex, pdftex, colorlinks, allcolors=blue]{hyperref}

\renewcommand{\le}{\leqslant}

\renewcommand{\ge}{\geqslant}

\newcommand{\eps}{\varepsilon}

\newcommand{\N}{{\mathbb N}}
\newcommand{\K}{{\mathbb K}}

\newcommand{\R}{{\mathbb R}}
\newcommand{\C}{{\mathbb C}}

\newcommand{\F}{\mathfrak{F}}
\newcommand{\FR}{\mathfrak{Fr}}

\newcommand{\U}{\mathfrak{U}}

\newcommand{\Oo}{\mathfrak{N}}

\newcommand{\Cauchy}{{\text{Cauchy}}}
\newcommand{\Compl}{{\text{Compl}}}

\newcommand{\supp}{{\text{supp}}}

\newcommand{\tvs}{topological vector space}

\newcommand{\TFAE}{the following assertions are equivalent: }

\theoremstyle{plain}
\newtheorem{theorem}{Theorem}[section]
\newtheorem{lemma}[theorem]{Lemma}
\newtheorem{corollary}[theorem]{Corollary}
\newtheorem{proposition}[theorem]{Proposition}

\theoremstyle{remark}
\newtheorem{remark}[theorem]{Remark}
\theoremstyle{definition}
\newtheorem{definition}[theorem]{Definition}
\newtheorem{problem}[theorem]{Problem}
\numberwithin{equation}{section}

\begin{document}
	\title[Completeness in topological vector spaces]{Completeness in topological vector spaces and filters on $\N$}
	\author{Vladimir Kadets}
	\address{School of Mathematics and Computer Sciences V.N. Karazin Kharkiv  National University,  61022 Kharkiv, Ukraine \newline
\href{http://orcid.org/0000-0002-5606-2679}{ORCID: \texttt{0000-0002-5606-2679}}}
	\email{v.kateds@karazin.ua}
	\author{Dmytro Seliutin}
	\address{School of Mathematics and Computer Sciences V.N. Karazin Kharkiv  National University,  61022 Kharkiv, Ukraine \newline
\href{https://orcid.org/0000-0002-4591-7272}{ORCID: \texttt{0000-0002-4591-7272}}}
	\email{selyutind1996@gmail.com}
	\thanks{ The research was partially supported by the National Research Foundation of Ukraine funded by Ukrainian state budget in frames of project 2020.02/0096 ``Operators in infinite-dimensional spaces:  the interplay between geometry, algebra and topology''}
	
\begin{abstract}
We study completeness of a topological vector space with respect to different filters on $\N$. In the metrizable case all these kinds of completeness are the same, but in non-metrizable case the situation changes. For example, a space may be complete with respect to one ultrafilter on $\N$, but incomplete with respect to another. Our study was motivated by [Aizpuru, List\'{a}n-Garc\'{i}a and Rambla-Barreno;  Quaest. Math., 2014] and  [List\'{a}n-Garc\'{i}a; Bull. Belg. Math. Soc. Simon Stevin, 2016] where for normed spaces the equivalence of the ordinary completeness and completeness with respect to $f$-statistical convergence was established.
\end{abstract}
	
\maketitle
	
\section{Introduction}

An increasing continuous function $f : [0, \infty) \to [0, \infty)$ is called a \emph{modulus function} if $f(0) = 0$ and $f(x+y) \le f(x) + f(y)$ for all $x, y \ge 0$. 
The $f$-density of a subset $A \subset \N$ is the quantity
$$
d_f (A) = \lim_{n \to \infty} \frac{f( |A \cap \overline{1, n}|)}{f(n)},
$$
where  $\overline{1, n}$  denotes the set of integers of the form $\{1,2, \ldots, n\}$ and the symbol $|D|$ means the number of elements in the set $D$. If for a set $A$ the above limit does not exist, then the $f$-density of $A$ is not defined.

Let $f$ be an unbounded modulus function, and $(x_n)$ be a sequence in a normed space $X$. An element $x \in X$ is called the \emph{$f$-statistical limit} of  $(x_n)$, if 
$$
d_f\left(\{n \in \N: \|x_n - x\| > \eps\}\right) = 0
$$ 
for every $\eps > 0$. 

According to \cite[Definition 3.2.]{aizpuru} $(x_n) \subset X$ is said to be \emph{$f$-statistically Cauchy} if for every $\eps > 0$ there exists $m \in \N$ such that 
$$
d_f\left(\{n \in \N: \|x_n - x_m\| > \eps\}\right) = 0.
$$
In the particular case of  the modulus function $f(t) = t$, the above definitions give the classical notions of \emph{statistical convergent} and \emph{statistical Cauchy} sequences, that are quite popular subjects of study.

Let us say that a normed space $X$ is \emph{$f$-complete}, if every $f$-statistically Cauchy sequence $(x_n) \subset X$ is $f$-statistically convergent.
Our research is motivated by \cite[Theorem 2.4]{listan-garcia} (see also \cite[Theorem 3.3]{aizpuru}): 
Let $X$ be a normed space. The following are equivalent:
(1) $X$ is complete; (2) $X$ is $f$-complete for every unbounded modulus $f$;
(3) there exists an unbounded modulus $f$ such that $X$ is $f$-complete.

Taking in account that convergence with respect to an unbounded modulus $f$ is equivalent to convergence with respect to the filter $\F_{f-st}$ of those subsets $B \subset \N$ that $d_f (\N \setminus B) = 0$, the above theorem leads to the natural question whether the same result is true in more general setting of filter convergence. We show that the answer is positive, and moreover it easily generalizes to metrizable  topological vector spaces. On the other hand, an attempt to generalize it further to arbitrary Hausdorff topological vector spaces fails because sequential completeness does not imply completeness in general. This motivates some results and leads to many open questions which we discuss at the end of our article.

Below, we use the term ``topological vector space'' (abbreviation TVS) for a Hausdorff topological vector space over the field $\K$ which is either the field $\R$ of reals, or the field $\C$ of complex numbers. We follow notation from \cite{kadets}, in particular for a topological space $X$, $\mathfrak{N}_z$ or $\mathfrak{N}_z(X)$ denotes the family of neighborhoods of point $z\in X$. If $X$ is a TVS, $\mathfrak{N}_0$ or $\mathfrak{N}_0(X)$ is the family of neighborhoods of zero,  $X^*$ is the set of all continuous linear functional on $X$ and  $X'$ is the set of ALL linear functional on $X$. For two subsets $A$, $B$ of a linear space the symbol $A+B$ denotes the corresponding Minkowskii sum: $A+B = \{a+b : a \in A, b \in B \}$. We refer to  \cite[Section 16.1]{kadets} for a short introduction to filters and ultrafilters, and to \cite{Bourbaki} for a detailed one. The very basic facts about  topological vector spaces can be found in \cite[Chapters 16 and 17]{kadets}, and for a much deeper exposition we refer to the classical book \cite{koethe}.

The structure of the paper is as follows. In the next section we recall, for  the reader's convenience, the definitions and basic facts about filters and filter convergence in topological spaces. After that, in the section ``Completeness, sequential completeness, and completeness over a filter on $\N$'', we recall the basic facts about Cauchy filters and completeness  in TVS, introduce formally the completeness over a filter on $\N$, list some features of this new property,  and deduce,  for general filters on $\N$ and a metrizable  TVS, the validity of equivalences like those in  \cite[Theorem 2.4]{listan-garcia}. After that we pass to the general non-metrizable case (Section ``Various types of completeness and classes of filters and spaces''). We discuss the relationship between completeness, sequential completeness, and completeness with respect to various filters on $\N$ (subsection ``Countable completeness''),  demonstrate a non-metrisable version  of  \cite[Theorem 2.4]{listan-garcia} in locally convex spaces under the additional boundedness condition for the Cauchy sequences in question (subsection ``Completeness and boundedness''), give an example of a sequentially complete space which is not complete with respect to ANY free ultrafilter on $\N$, and  give an example of a space which is complete with respect to a free ultrafilter on $\N$ but is not complete with respect to some other ultrafilter (subsection ``Completeness and ultrafilters''). We conclude the paper with some open questions.

\section{Basic facts about filters and filter convergence}
	
Let $\Omega$ be a non-empty set. Recall that \textit{filter} on $\Omega$ is a non-empty family $\F$ of subsets in $\Omega$, satisfying the following axioms: $\Omega \in \F$; $\emptyset \not\in \F$; if $A,\ B \in \F$ then $A \cap B$; and if $A \in \F$ and $D \supset A$ then $D \in \F$.

Every point  $x_0 \in \Omega$ generates the \emph{trivial filter} of all subsets containing  $x_0$. The \textit{Fr\'{e}chet filter}  $\FR = \{A \subset \N: |\N \setminus A| < \infty\}$ is the simplest example of non-trivial filter on $\N$.

A subset $A \subset \Omega$ is called  \emph{$\F$-stationary} if $A$ intersects all elements of $\F$.
	
A non-empty family  $G \subset 2^{\Omega}$ is called a \textit{filter base}, if $\emptyset \notin G$ and for every pair $A, B \in G$ there exists $C \in G$ such that $C \subset A \cap B$. The \emph{filter generated by the base} $G$ is the collection of all those $A \subset \Omega$ for which there is a $B \in G$ such that $A \supset B$.  A filter $\F$ is generated by a base  $G$ if  $G \subset \F$ and each element of $\F$ contains at least one element from $G$. When we write $\F=\F(G)$ it means that $G$ is a base for the filter $\F$. In this notation, the trivial filter on $\Omega$ generated by $x_0 \in \Omega$ is equal to $\F(\{\{x_0\}\})$, and   $\FR = \F(\{\N \setminus \overline{1, n}\}_{n \in \N})$.

The set of all filters on $\Omega$ is naturally ordered by inclusion. Maximal in this ordering filters are called \emph{ultrafilters}. The only constructive examples of ultrafilters are the trivial ones, but the Zorn's lemma implies the existence of many non-trivial ultrafilters: for every filter $\F$ on $\Omega$ there is an ultrafilter $\U$ such that $\F \subset \U$. In particular, on $\N$ there are \emph{free} ultrafilters, i.e. ultrafilters that dominate the Fr\'{e}chet filter.  Below, if the contrary is not precised, on $\N$ we consider only free filters and ultrafilters. 

Let $\Omega$ be a set with a filter $\F_0$, $Y$ be another set, and $f : \Omega \to Y$ be a function. The natural collection $f(\F_0) = \{f(A) \colon A \in \F_0\}$ in $Y$ is not necessarily a filter, but is a filter base. By this reason the  \emph{image of the filter} $\F_0$ is defined as $f[\F_0] := \F(f(\F_0))$. 

A sequence $x = (x_n) \subset Y$ is a function that acts from $\N$ to $Y$. By this reason at our convenience we use notation $x(n)$ for $x_n$,  $x(A)$ for the set $\{x_n : n \in A\}$, etc.

Let $Y$ be a topological space, $\F$ be a filter on $Y$. A point $z \in Y$ is called a \textit{limit} of the filter $\F$ ($z = \lim \F$), if $\Oo_z \subset \F$, and is called a  \textit{cluster point} of $\F$ if every neighborhood of $z$ is $\F$-stationary. In a Hausdorff space the limit of $\F$, if exists, is the unique cluster point of $\F$.

Let $\F_1 \subset \F_2$ be filters on $Y$. Then every cluster point of $\F_2$ is a cluster point of $\F_1$, and the limit of $\F_1$, if exists, is the limit of $\F_2$.

A sequence $x = (x_n) \subset Y$ is called \textit{converging to an element $y \in Y$ over filter $\F$ on $\N$} ($y = \lim_{\F} x_n$), if  $y = \lim x[\F]$, that is for every $V \in \Oo_y$ there exists such $A \in \F$ that $x(A) \subset V$. $y$ is a \emph{cluster point of $x$ over $\F$} if $y$ is a cluster point of $x(\F)$. The huge advantage of ultrafilters, that we use in some instances below, is that for every ultrafilter $\U$ on $\N$ every sequence with values in a compact (in particular every bounded numerical sequence) possesses a limit over $\U$. 

\section{Completeness, sequential completeness, and completeness over a filter on $\N$}
	
Let $X$ be a TVS. A filter $\F$ on $X$ is called \textit{Cauchy filter}, if for every $U \in \Oo_0$ there exists $A \in \F$ such that $A - A \subset U$ (write $\F \in \Cauchy$).
Evidently, if $\F_1 \subset \F_2$ are filters on $X$ and  $\F_1 \in \Cauchy$, then  $\F_2 \in \Cauchy$.
	
A \tvs\ $X$ is called \textit{complete} ($X \in \Compl$), if every Cauchy filter on $X$ has a limit. Remark, that the most important examples of normed spaces are complete. This is the reason why  in frames of normed spaces the majority of researchers are concerned only about complete (i.e. Banach) spaces. For the topological vector spaces the situation is very different: many spaces that motivated the whole theory, like infinite-dimensional dual Banach spaces equipped with the weak$^*$ topology, are incomplete ($X^*$ is not closed in $X{'}$ in the pointwise convergence topology), so one cannot avoid them in frames of the general theory.

\begin{proposition}[{\cite[Section 16.2.2, Theorem 2]{kadets}}] \label{rem-clust-lim}
Let $X$ be a TVS,  $\F$ be a Cauchy filter $\F$ on $X$ and $z \in X$ be a cluster point of $\F$, then $z = \lim \F$. 
\end{proposition}
	
A sequence $x = (x_n) \in X^\N$ is said to be a \textit{Cauchy sequence over  filter $\F$ on $\N$}, if $x[\F]$ is a Cauchy filter on $X$. We denote the last property by $x \in \Cauchy(\F)$.	 In other words,   $x \in \Cauchy(\F)$, if for every $U \in \Oo_0$ there exists $B \in \F$ such that $x(B) - x(B) \subset U$.

$x = (x_n) \in X^\N$ is said to be a \textit{Cauchy sequence} if  $x \in \Cauchy(\FR)$. In other words,  $x \in \Cauchy(\FR)$   if for every $U \in \Oo_0$ there exists $N \in \N$ such that $x_n - x_m \in U$ for all $n, m \ge N$.
	
It seems to us that the following definition, which is the main object of study in this article, is new. At least, we did not find it in the literature.
	
\begin{definition} \label{def-F-complete}
Let $\F$ be a free filter on $\N$. A \tvs\ $X$ is said to be \textit{complete over $\F$}, if every Cauchy sequence over $\F$ in $X$ has a limit over $\F$. We denote this property by $X \in \Compl(\F)$.
\end{definition}

Recall, that a \tvs \ $X$ is called \textit{sequentially complete}, if $X \in \Compl(\FR)$, that is, if every Cauchy sequence in $X$ has a limit.

Let us list some elementary general facts about completeness over filters.

\begin{theorem} \label{thm-ComplF1F2}
$ \ $

\begin{enumerate}[\emph{(}1\emph{)}]
\item If $X \in \Compl$, then $X \in \Compl(\F)$ for every free filter $\F$ on $\N$. In particular,

\item  a complete TVS is sequentially complete.

\item In order to verify that $X \in \Compl(\F)$ it is sufficient to check that every Cauchy sequence over $\F$ in $X$ has a cluster point over $\F$.

\item  If $\F_1 \subset \F_2$ are filters on $\N$ and $X \in \Compl(\F_2)$, then  $X \in \Compl(\F_1)$

\item  If $\F$ is a filter on $\N$, $f: \N \to \N$ is a function, and  $X \in \Compl(\F)$, then $X \in \Compl(f[\F])$.

\item  If $\F$ is a filter on $\N$, $f: \N \to \N$ is an injective function, and   $X \in \Compl(f[\F])$, then $X \in \Compl(\F)$.

\end{enumerate}
\end{theorem}

\begin{proof} (1) follows from the definition, (2) is a particular case of (1) for $\F = \FR$. 

Let us check (3).  Let $\F$ be a free filter on $\N$ and $x = (x_n) \in X^\N$  be a Cauchy sequence over $\F$. Assume that we know that $x$  has a cluster point over $\F$. This means that  a Cauchy filter $x[\F]$ has has a cluster point, so the application of Proposition \ref{rem-clust-lim} gives us the existence of $\lim x[\F]$.

Now it is turn of the statement (4). If  $x = (x_n) \in X^\N$  is a Cauchy sequence over $\F_1$, then $x \in \Cauchy(\F_2)$. Consequently, by $\F_2$-completeness of $X$ there is $y \in X$ such that $y = \lim x[\F_2]$. Taking in account that  $ x[\F_2] \supset  x[\F_1]$, this $y$ is a cluster point for $ x[\F_1]$. It remains to apply the statement (3).

Let us demonstrate (5). Let  $x = (x_n) \in X^\N$  be a Cauchy sequence over $f[\F]$. Consider the sequence $y = x \circ f$, i.e. $y = (y_n)$, where $y_n = x_{f(n)}$. Then $y[\F] = x[f[\F]]$, so $y$ is a Cauchy sequence over $\F$. By the $\F$-completeness assumption, there exist $\lim_\F y$ which  is the limit of $x$ over $f[\F]$.

Finally, let us demonstrate (6). Denote $g: \N \to \N$ a left inverse to $f$, which means that $g$ satisfies the condition $g(f(n)) = n$ for all $n \in \N$.  Let  $x = (x_n) \in X^\N$  be a Cauchy sequence over $\F$. Consider the sequence $y = x \circ g$. Then $y \circ f = x \circ g \circ f = x$. So, $y \circ f \in \Cauchy(\F)$ which means that  $y [f[\F]]$ is a Cauchy filter, so $y \in \Cauchy(f[\F])$. By the $f[\F]$-completeness assumption, there exist $\lim_{f[\F]} y$. By the definition, this means that $y [f[\F]] = x[g[f[\F]]] = x[\F]$ is a convergent filter.
\end{proof}	

Remark, that item (5)  of the previous statement is of interest for us only if $f[\F] \supset \FR$, which is not always the case. Also, in (6)  the assumption of injectivity cannot be omitted because, without it, it may happen that  $f[\F]$ is a trivial filter, in which case the   $f[\F]$-completeness is true for every space but does not give any information about the $\F$-completeness. 

Now we are ready to the promised extension of \cite[Theorem 3.3]{aizpuru} to general filters. 

\begin{theorem} \label{thm-metriz-compl}
Let $X$ be a TVS possessing a countable base of neighborhoods of zero (in other words, $X$ is metrizable), then \TFAE
\begin{enumerate}[\emph{(}i\emph{)}]

\item $X$ is complete. 

\item $X \in \Compl(\F)$ for every free filter $\F$ on $\N$. 

\item There is a every free filter $\F$ on $\N$ such that  $X \in \Compl(\F)$

\item $X$ is sequentially complete.

\end{enumerate}
\end{theorem}

\begin{proof}
The implication (i)$\Rightarrow$(ii) is covered by item (1) of Theorem \ref{thm-ComplF1F2}, the implication (ii)$\Rightarrow$(iii) is evident, and the implication (i)$\Rightarrow$(ii) follows from the item (4) of Theorem \ref{thm-ComplF1F2}. It remains to demonstrate that (iv)$\Rightarrow$(i). This fact is well-known (\cite[Section 16.2.2, Exercise 4]{kadets}) and may be deduced from an analogous theorem for uniform spaces. Nevertheless, for the reader's convenience (and for a reference below) we prefer to give a direct proof. So, let $U_n \in \Oo_0$, $U_1 \supset U_2 \supset \ldots$ be a base of  neighborhoods of  zero with the property that $U_{n+1} + U_{n+1} \subset U_n$, and let  $\F$ be a Cauchy filter on $X$. For each $n \in \N$ pick $A_n \in \F$ such that $A_n - A_n \subset U_n$, and select an $x_n \in \bigcap_{k=1}^n A_k$. Then  $x = (x_n)$ is a Cauchy sequence. Indeed, for every $U \in \Oo_0$ there is an $N \in \N$ such that $U_N \subset U$. Then, for $n, m \ge N$ we have that $x_n - x_m \in A_N - A_N \subset U_N \subset U$.

Since $x = (x_n)$ is a Cauchy sequence and $X$ is sequentially complete, there is $y:= \lim_{n \to \infty} x_n$. Let us show that the same $y$ is the limit of $\F$. Consider an arbitrary neighborhood $V \in \Oo_y$. $V - y \in \Oo_0$, so there is $m \in \N$ such that $U_{m} \subset V - y$. By the definition of $y$, there is $k > m$ such that $x_k \in U_{m+1} + y$. For this $k$ we have $A_k - x_k \subset A_k - A_k \subset U_k \subset U_{m+1}$, consequently $A_k \subset U_{m+1} + x_k$. This means that 
$$
V \supset U_m + y \supset U_{m+1} + U_{m+1} + y  \supset U_{m+1} + x_k \supset A_k,
$$
so $V \in \F$.
\end{proof}

\section{Various types of completeness and classes of filters and spaces}

Although in metrizable spaces all types of completeness that we mentioned above are the same, in general non-metrizable spaces the picture is much more complex. It is well-known that an incomplete  topological vector space may be sequentialy complete. The most important example of such kind is the Hilbert space $\ell_2$ equipped with the weak topology. This section is devoted to the non-metrizable case, where many interesting examples come from the duality theory for locally convex spaces. A very good comprehensive introduction to duality is \cite{R-R}, a shorter one may be found in \cite[Chapters 17, 18]{kadets}.  As usual, for a duality pair $X, Y$ we denote $\sigma(X, Y)$  the weak topology on $X$ generated by $Y$.

\subsection{Countable completeness}

\begin{definition}
A \tvs \ $X$ is said to be \textit{countable complete}, if $X \in \Compl(\F)$ for all $\F$ on $\N$.
\end{definition} 

We are going to show that countable completeness implies completeness for separable spaces, but does not imply completeness in general, and that sequential completeness does not imply countable completeness. At first, an easy reformulation.

\begin{lemma}\label{lem-count-compl-1}
For a  \tvs \ $X$ \TFAE
\begin{enumerate}[\emph{(}i\emph{)}]

\item $X$ is countable complete. 

\item For every Cauchy filter $\F$ on $X$, if $\F$ has a countable element then $\F$ has a limit. 
\end{enumerate}
\end{lemma}
	
\begin{proof}
(ii)$\Rightarrow$(i). Let $\F$ be a filter on $\N$, and $x = (x_n) \in 2^X$ be an $\F$- Cauchy sequence in $X$. Then $x[\F]$  is Cauchy filter on $X$,  $x[\F]$ has a countable element  $x(\N)$, so $x[\F]$ has a limit, which, according to the definition means that $x$ has limit with respect to $\F$.

(i)$\Rightarrow$(ii)   Let $\F$ be a non-trivial Cauchy filter on $X$,  $A \in \F$ be a countable element. Let $x: \N \to A$ be a bijection. Define $x^{-1}(\F) = \{D \subset \N : x(D) \in \F\}$. Then $x[x^{-1}(\F)] = \F$, so $x \in \Cauchy(x^{-1}(\F))$  which means the existence of $\lim_{x^{-1}(\F)}x$ which, by the definition, is the limit of $x[x^{-1}(\F)] = \F$ in $X$.
\end{proof}

\begin{definition}
A \tvs \ $X$ is said to be \textit{ asymptotically countable}, if for every Cauchy filter $\F$ on $X$ there is a countable set $A \subset X$ such that $A \cap (B + V) \neq \emptyset$ for every $V \in \Oo_0$ and $B \in \F$.
\end{definition} 	

Evidently, a separable space is asymptotically countable. Remark that there are non-separable asymptotically countable spaces. A funny example comes from the fact that every complete space is asymptotically countable (just in the notation from the above definition take such $A$ that $\lim {\F} \in A$). Less evident examples come from  asymptotic countability of every TVS that has a countable base of zero neighborhoods: this can be done similarly to the implication (iv)$\Rightarrow$(i) of Theorem \ref{thm-metriz-compl}.

\begin{theorem}\label{thm-count-comp-sep}
For an asymptotically countable \tvs \ $X$ (in particular, for separable $X$) its completeness is equivalent to its countable  completeness.
\end{theorem}

\begin{proof} If $X$ is complete, then it is complete with respect to all filters on $\N$ by the evident item (1) of Theorem \ref{thm-ComplF1F2}, so we only need to check the inverse implication.

Let $X$ be  asymptotically countable and countable complete. Consider  a non-trivial Cauchy filter $\F$ on $X$. Fix  a corresponding countable set $A \subset X$ such that the collection $G \subset 2^X$ consisting of all sets of the form $A \cap (V + B)$, where $V \in \Oo_0$  and $B \in \F$, does not contain the empty set. Evidently, $G$ is a filter base. Denote $\widetilde \F$ the filter generated by the base $G$. Since $A \in G \subset \widetilde \F$, $\widetilde \F$ has a countable element. Also, $\widetilde \F \in \Cauchy$.  Indeed, let  $U \in \Oo_0$. Select a balanced neighborhood  $V \in \Oo_0$ such that $V + V + V \subset U$. We know that $\F \in \Cauchy$, consequently there exists $B \in \F$ with $B - B \subset V$. Then $A \cap (V + B) \in \widetilde \F$ and
$$
(A \cap (V + B)) - (A \cap (V + B)) \subset  (V + B) -  (V + B) 
$$
$$
\subset  (V - V) +  (B - B) \subset V - V + V = V + V + V  \subset U,
$$
which completes the proof of the fact that $\widetilde \F \in \Cauchy$. Then, by the countable completeness, there is $y \in X$ such that $y = \lim \widetilde \F$ (we use (ii) of Lemma \ref{lem-count-compl-1}). It remains to show that $y = \lim \F$. In order to demonstrate this, it is sufficient to show that $y$ is a cluster point for $\F$ (Proposition \ref{rem-clust-lim}). For this, let us consider an arbitrary neighborhood $U \in \Oo_y$, arbitrary $D \in \F$ and demonstrate that $U \cap D \neq \emptyset$. Select a balanced neighborhood  $V \in \Oo_0$ such that $V + V  \subset U - y$. Since   $y = \lim \widetilde \F$, we have that  $V + y \in \widetilde \F$.  Consequently, $V + y$  contains a subset of the form $A \cap (W + B)$, where $B \in \F$, $W \in \Oo_0$. We have that
$$
U \supset V + V +y  \supset   V + A \cap (W  + B) \supset  V + A \cap (W \cap V + B\cap D).
$$
Take a point $a \in A \cap (W \cap V + B\cap D)$. It can be written in the form $a = v + d$, where $v \in V$, $d \in D$. Then,
$$
d = a - v \in A \cap (W \cap V + B\cap D) + V \subset U,
$$
so $U \cap D \neq \emptyset$.
\end{proof}

\begin{corollary}\label{cor-example-seq-noncount}
The sequential completeness does not imply the countable completeness.
\end{corollary}

\begin{proof} 
By the previous theorem, in order to get an example of such a kind it is sufficient to find an incomplete separable TVS which is sequentially complete. The classical example for this is the Hilbert space $\ell_2$ equipped with the weak topology. More generally, for every separable infinite-dimensional Banach space $X$ the dual space $X^*$  in the topology $\sigma(X^*, X)$ is sequentially complete (see Theorem \ref{thm-bcompl-Fcoml} below for a stronger result), separable but incomplete. 
\end{proof}

The next result shows that the asymptotic countability assumption in Theorem \ref{thm-count-comp-sep} cannot be omitted.

\begin{theorem}\label{thm-example-non-comp-countcomp}
There exists a non-complete TVS \ $X$ of continuum cardinality which is countably complete.
\end{theorem}

\begin{proof} 
Consider $\R^{[0, 1]}$ -- the space of all functions $f: [0, 1] \to \R$ equipped with the standard product topology, i.e. the topology of pointwise convergence. The space $X$ we are looking for will be the subspace of  $\R^{[0, 1]}$ consisting of functions with countable support. In other words,  $f: [0, 1] \to \R$ lies in $X$ if the set $\supp f := \{t \in [0, 1]: f(t) \neq 0\}$ is at most countable. Since $X$ is a dense proper subspace of  $\R^{[0, 1]}$, it cannot be complete (a complete subspace of any TVS is closed \cite[Section 16.2.2, Theorem 4]{kadets}). Let us demonstrate that  $X$ is countably complete.

Let $\F$ be a filter on $\N$, and $x = (x_n) \in 2^X$ be an $\F$- Cauchy sequence in $X$. Then $x$ is  $\F$- Cauchy as a sequence in  $\R^{[0, 1]}$. By the completeness of  $\R^{[0, 1]}$, there is $f \in \R^{[0, 1]}$ such that $f = \lim_\F x_n$ in  $\R^{[0, 1]}$.
\end{proof}

\subsection{Completeness and boundedness}

In the very recent paper \cite{bondt} Ben De Bondt and Hans Vernaeve introduced several concepts that are very useful for our study. Below we present the most important for us particular case.

Let $X$ be a Banach space, $(x_n^*) \subset X^*$ be a sequence of functionals, and $\F$ be a free filter on $\N$. The sequence $(x_n^*) $ is said to be \emph{ pointwise $\F$-bounded}, if for every $x \in X$ there is a $C = C(x) > 0$ such that $\{n \in \N: |x_n^*(x)| < C \} \in \F$. The sequence $(x_n^*) $ is said to be  \emph{$\F$-bounded} (\emph{stationary $\F$-bounded}), if there is a $C > 0$ such that $\{n \in \N: \|x_n^*\| < C \} \in \F$ ($\{n \in \N: \|x_n^*\| < C \}$ is $\F$-stationary). 

A free filter  $\F$ on $\N$ is called a \emph{B-UBP-filter} (\emph{stationary B-UBP-filter}), if for every Banach space $X$ every pointwise $\F$-bounded sequence  $(x_n^*) \subset X^*$ is $\F$-bounded. This property is weaker (at least formally) than the property of being \emph{(stationary) Banach-UBP-filter}, for which the authors of  \cite{bondt} demanded a similar statement for linear continuous operators from $X$ to arbitrary locally convex space $Y$ with $\F$-equicontinuity instead of $\F$-boundedness in the conclusion.

The fact that the $\FR$ is  B-UBP is just the classical Banach-Steinhaus theorem. On the other hand, many classical filters $\F$ do not enjoy this property, because of the existence of $\F$-unbounded pointwise $\F$-convergent sequences in dual Banach spaces. The latter effect was remarked in
\cite[Theorem 1]{conkad} for the statistical convergence and was investigated in detail in \cite{gakad}, \cite{Kad-cyl}, and \cite{KLO2010}.

Ben De Bondt and  Hans Vernaeve presented non-trivial descriptions and examples of  Banach-UBP filters and stationary Banach-UBP filters and demonstrated that the existence of of B-UBP ultrafilters is consistent in the standard ZFC axiom system.

This motivates the following definition.

\begin{definition} \label{def-F-complete}
Let $\F$ be a free filter on $\N$. A \tvs\ $X$ is said to be \textit{boundedly complete over $\F$}, if every bounded Cauchy sequence over $\F$ in $X$ has a limit over $\F$. We denote this property by $X \in \Compl_b(\F)$. $X$ is said to be \textit{boundedly countably complete} if it is boundedly complete  over all filters $\F$ on $\N$.
\end{definition}

The next theorem gives a plenty of examples.

\begin{theorem}\label{thm-dual-bcompl}
Let $Y$ be a Banach space, then $(Y^*, \sigma(Y^*, Y) )$ is boundedly countably complete.
\end{theorem}
\begin{proof}
The proof repeats almost literally the demonstration of the well-known facts \cite[Section 6.4.3, Theorems 1 and 2]{kadets} about the ordinary pointwise convergence. Namely, let $(x_n^*) \subset Y^*$ be a bounded sequence. Recall, that $\sigma(Y^*, Y)$-boundedness is equivalent to the boundedness in norm (Banach-Steinhaus), so $\sup_n\|x_n*\| = C < \infty$. Assume that for some filter $\F$ on $\N$ the sequence $(x_n^*)$ is $\F$-Cauchy in topology $\sigma(Y^*, Y)$. This means that for every $x \in Y$ the sequence  $(x_n^*(x)) \subset \K$ is $\F$-Cauchy. By completeness of $\K$ (which is either $\R$ or $\C$),   $\lim_{\F} x_n^*(x)$ exists for all $x\in Y$.  Consider the map $f \colon Y \to \K$ given by the recipe $f(x) = \lim_{\F} x_n^*(x)$. At first, it is a linear functional. Indeed,
$ f(ax_1 + bx_2) = \lim_{\F} x_n^* (ax_1 + bx_2) = a \lim_{\F} x_n^* (x_1) + b \lim_{\F} x_n^* (x_2) = af(x_1) + f(x_2)$.
At second, the estimate $|f(x)| = \lim_{\F} \|x_n^*(x)\| \le C\|x\|$, which holds for all $x\in Y$, demonstrates that $f$ is continuous, so $f \in Y^*$, and $f = \lim_{\F} x_n^*$ in  $\sigma(Y^*, Y)$.
\end{proof}

\begin{theorem}\label{thm-bcompl-Fcoml}
Let $Y$ be a Banach space, $\F$ be a stationary B-UBP-filter on $\N$, then $(Y^*, \sigma(Y^*, Y) )$ is $\F$-complete.
\end{theorem}
\begin{proof}
Let $x^* = (x_n^*) \subset Y^*$ be $\F$-Cauchy in the topology $\sigma(Y^*, Y)$. Then for every $x \in Y$ the sequence  $(x_n^*(x)) \subset \K$ is $\F$-Cauchy, which implies that $x^*$ is pointwise $\F$-bounded. From the definition of stationary B-UBP-filter we deduce the existence of an $\F$-stationary set $A \subset \N$ such that $\sup_{n \in A} \|x_n^*\| < \infty$. Consider the collection $G$ all sets of the form $A\cap B$, $B \in \F$. $G$ is a filter base. Denote $\F_A$ the filter on $\N$ generated by this particular base $G$.  Let $g : \N \to A$ be a bijection. Denote $\F_g$ the filter of all those $B \subset \N$ for which $g(B) \in G$. Finally, consider $y^* = x^* \circ g$. Then $y^* $ is pointwise bounded and is $\F_g$-Cauchy. By Theorem \ref{thm-dual-bcompl} the sequence  $y^*$ is pointwise $\F_g$-convergent to some $f \in Y^*$. This means that the sequence $x^*$ converges to $f$  with respect to the filter $g[\F_g] = F_A$. By the construction,  $\F_A \supset \F$, consequently,  $f$  is an $\F$-cluster point for $x^* = (x_n^*)$  in $\sigma(Y^*, Y)$.  But $x^*$ is $\F$-Cauchy, so its cluster point $f$ is its limit (Proposition \ref{rem-clust-lim}).
\end{proof}

We don't know whether for general TVS the sequential completeness implies $f$-statistical completeness for every unbounded modulus $f$, but we have an analogous result for bounded completeness in locally convex spaces.

Recall that in the particular case of  the modulus function $f(t) = t$,  $f$-statistical convergence reduces to the well-known statistical convergence, which is generated by the filter $\F_{st}$, whose elements are those $A \subset \N$, for which
\begin{equation} \label{eq:stat-conv}
\lim_{n \to \infty} \frac{ |A(n)|}{n} = 1,
\end{equation}
where $A(n) = A \cap \overline{1, n}$. First, we need a generalization of the following fact that was remarked already in \cite{fast}: if a bounded numerical sequence $(x_n)$ converges statistically to a number $a$, then it is Cesaro convergent to $a$, i.e.
$$
\lim_{n \to \infty} \frac{1}{n}\sum_{k=1}^n x_k = a.
$$
\begin{lemma}\label{lem-loc-conv-stat-Cesaro}
Let $X$ be a locally convex TVS and $x = (x_n) \in 2^X$ be a bounded $\F_{st}$-Cauchy sequence in $X$, then the sequence   $y = (y_n) \in 2^X$, where  $y_n = \frac{1}{n}\sum_{k=1}^n x_k$, is a bounded Cauchy sequence.
\end{lemma}
\begin{proof}
Let $U \in \Oo_0$ be an open balanced convex neighborhood of zero and $p$ be the seminorm whose open unit ball is equal to $U$. Denote $C = \sup_n p(x_n)$. Then all $x_k \in C U$ and, by convexity, all $y_n \in C U$, which proves the boundness of $y$. It remains to show that $y$ is a Cauchy sequence.

According to our assumption, there is a set $A \subset \N$ that satisfies \eqref{eq:stat-conv} and such that $p(x_n - x_m) <  \frac{1}{2}$ for all $m, n \in A$
Select an $N \in \N$ in such a way that for all $n \ge N$
$$
\frac{ |(\N \setminus A)(n)|}{n} < \frac{1}{8C}.
$$
Then, for $n, m \ge N$ we have

\begin{align*}
&p\left( y_n - y_m \right) = p\left(  \frac{1}{n}\sum_{k=1}^n x_k -  \frac{1}{m}\sum_{j=1}^m x_j \right) = p\Bigl( \frac{1}{n}\sum_{k \in (\N \setminus A) (n)} x_k  \\ 
&-  \frac{1}{m}\sum_{j \in (\N \setminus A) (m)} x_j   +  \frac{1}{|A(n)| |A(m)|}\sum_{k \in  A(n), \ j \in  A (m)} (x_k - x_j)  \\
&-   \left(\frac{1}{|A(n)| } - \frac1n \right)\sum_{k \in  A(n)}x_k - \left(\frac{1}{|A(m)| } - \frac1m \right)\sum_{j \in  A(m)}x_j)\Bigr) 
\end{align*} 
\begin{align*}
&\le  \frac{1}{n}\sum_{k \in (\N \setminus A) (n)} p(x_k) 
+  \frac{1}{m}\sum_{j \in (\N \setminus A) (m)} p(x_j)   \\
&+  \frac{1}{|A(n)| |A(m)|}\sum_{k \in  A(n), \ j \in  A (m)} p(x_k - x_j)  \\
&+   \left(\frac{1}{|A(n)| } - \frac1n \right)\sum_{k \in  A(n)}p(x_k) + \left(\frac{1}{|A(m)| } - \frac1m \right)\sum_{j \in  A(m)}p(x_j)  \\
& <  \frac{C}{n}  \frac{n}{8C}   +  \frac{C}{m}  \frac{m}{8C}  +   \frac{1}{2} +  \frac{C}{n}  \frac{n}{8C} +  \frac{C}{m}  \frac{m}{8C}    = 1,
\end{align*} 
that is $y_n - y_m \in U$.
\end{proof}

\begin{theorem}\label{thm-loc-conv-bFst-compl}
Let $X$ be a boundedly sequentially complete locally convex TVS, then $X \in \Compl_b(\F_{st})$.
\end{theorem}
\begin{proof}
Let $x = (x_n) \in 2^X$ be a bounded $\F_{st}$-Cauchy sequence in $X$. According to the previous lemma,  the sequence   $y = (y_n) \in 2^X$, where  $y_n = \frac{1}{n}\sum_{k=1}^n x_k$, is a bounded Cauchy sequence, so it has a limit in $X$. 

Denote $a = \lim_{n \to \infty} y_n \in X$. Let us demonstrate that $a = \lim_{\F_st} y$. To do this, consider $U$, $p$, $C$, $A$ and $N$ from the proof of Lemma \ref{lem-loc-conv-stat-Cesaro}. Also, fix such an $M > N$ that $p(y_n - a) < \frac14$ for all $n > M$. Then, for every $n \in A \setminus \overline{1, M}$ we have
$$
p(x_n - a) \le \frac14 + p(x_n - y_n) \le \frac14  + p\left(\frac{1}{n}\sum_{k \in (\N \setminus A) (n)} x_k \right)  + p\left( x_n  -  \frac{1}{n}\sum_{k \in  A(n)}  x_k  \right)   
$$
\begin{align*}
&\le \frac14 +  \frac{C}{n}  \frac{n}{8C} +  p\left( \frac{1}{|A(n)|}\sum_{k \in  A(n)}  (x_n  - x_k ) \right)
+  \left(\frac{1}{|A(n)|}- \frac{1 }{n}  \right) p\left(\sum_{k \in  A(n)}  x_k  \right) \\
&\le \frac14 +  \frac{C}{n}  \frac{n}{8C} +  \frac{1}{2} +  \frac{C}{n}  \frac{n}{8C} = 1,
\end{align*} 
that is $x_n - a \in U$.
\end{proof}

\begin{corollary}\label{cor-loc-conv-bFst-compl}
Let $X$ be a boundedly sequentially complete locally convex TVS, then $X \in \Compl_b(\F_{f-st})$ for every unbounded modulus function $f$.
\end{corollary}
\begin{proof}
Since $\F_{f-st} \subset \F_{st}$ (see the reasoning just before \cite[Corollary 2.2]{aizpuru}), it remains to apply the statement (4) of theorem \ref{thm-ComplF1F2} in its version for bounded sequences, which works the same way as the original one.  
\end{proof}

\subsection{Completeness and ultrafilters}

Let us start with an easy observation.
\begin{remark} \label{rem-all-ultraf}
If a TVS $X$ is complete with respect to all ultrafilters on $\N$, then $X$ is countably complete.  Select an ultrafilter $\U \supset \F$. By our assumption, $X \in \Compl(\U)$, and it remains to apply  (4) of Theorem \ref{thm-ComplF1F2} in order to show that  $X \in \Compl(\F)$.
\end{remark}

The above remark motivates some natural questions. At first, is it true that the completeness with respect to one ultrafilter implies the completeness with respect to all other ultrafilters (and hence implies the countable completeness)? If the answer is negative, then the second question arises: does the sequential  completeness imply completeness with respect to some ultrafilter? The negative answers to both questions are given below (for the first one the answer is given under an additional set-theoretic assumption).

\begin{theorem}\label{thm-bcompl-Fcoml}
Under the Martin's axiom there are free ultrafilters $\U_1, \U_2$ on $\N$ and a TVS $X$ such that $X \in \Compl(\U_1)$, but  $X \notin \Compl(\U_2)$.
\end{theorem}
\begin{proof}

Let $X = (Y^*, \sigma(Y^*, Y) )$, where $Y$ is a separable infinite-dimensional Banach space. According to \cite[Corollary 5.1 and Theorem 5.3]{bondt}, the Martin's axiom guaranties the existence of
$2^{2^{\aleph_0}}$
-many B-UBP-ultrafilters.  Let $\U_1$ be a B-UBP-ultrafilter on $\N$. Due to Theorem \ref{thm-bcompl-Fcoml}, $X \in \Compl(\U_1)$. On the  other  hand, $X$ is separable (the dual to a separable Banach space contains a countable total system \cite[Section 17.2.4, Corollary 2]{kadets} and, consequently, is w$^*$-separable) and incomplete, so by Theorem \ref{thm-count-comp-sep} $X$  is not  countably complete, which implies (Remark \ref{rem-all-ultraf}) that  $X \notin \Compl(\U_2)$ for some free ultrafilter $\U_2$ on $\N$.
\end{proof}

Recall, that the dual to  $\ell_1$ is  $\ell_\infty$, and for every  $x = (x_1, x_2, \ldots) \in \ell_\infty$ and  $y = (y_1, y_2, \ldots) \in \ell_1$  the action of $x$ on $y$ is $x(y) = \sum_{n \in \N}x_n y_n$.

\begin{theorem}\label{thm-seq-comp-notultrcomp}
The space $\ell_1$ in the weak topology $\sigma(\ell_1, \ell_\infty)$ is sequentially complete, but is not boundedly complete  over any  free ultrafilter $\U$ on $\N$.
\end{theorem}
\begin{proof}

Weak sequential completeness of $\ell_1$ (as well as of all spaces $L_1(\mu)$) is a classical Banach space theory result, see \cite[Theorem 2.5.10]{albiac}. Now, let us fix an arbitrary free ultrafilter $\U$ on $\N$ and demonstrate that $(\ell_1, \sigma(\ell_1, \ell_\infty)) \notin \Compl_b(\U)$. Denote $\left\{{e_n} \right\}_1^\infty$ the \emph{canonical basis} of $\ell_1$, that is $e_1 = (1,0,0,\ldots)$, $e_2 = (0,1,0,\ldots)$,\ldots . For every  $x = (x_1, x_2, \ldots) \in \ell_\infty$ the values $x(e_n) = x_n$ form a bounded sequence of scalars, hence there is the limit on $x(e_n)$ over $\U$. Consequently, $(x(e_n))$ is $\U$-Cauchy in the topology  $\sigma(\ell_1, \ell_\infty)$.

Now we show that the sequence $(e_k)$ does not have a weak limit over $\U$. Assume that there exists $z =(z_1, z_2, \ldots) \in \ell_1$ such that $x(z) = \lim_\U (x(e_n)) = \lim_\U (x_n)$ for all $x = (x_1, x_2, \ldots) \in \ell_{\infty}$. Then, on the one hand, considering $e_k$ as elements of $\ell_\infty$ we get that for every $k \in \N$
$$
z_k = e_k(z) = \lim_{\U, n} (e_k(e_n)) = 0,
$$
but on the other hand, taking $x = (1, 1, 1, \ldots)$ we get that
$\sum_{n \in \N} z_n =  \lim_\U (x_n) = 1$.
We came to a contradiction.
\end{proof}

The above result can be viewed in a bit different way. Consider $\ell_1$ as a subspace  of $\ell_\infty^*$. Then in $\sigma(\ell_\infty^*, \ell_\infty)$ we have that   $(e_k)$ is $\U$-convergent to the functional $x \mapsto \lim_\U x$, and this functional belongs to $\ell_\infty^* \setminus \ell_1$.

\section{Concluding remarks and open questions}

The following challenging problem remains open. 
\begin{problem} \label{prob1}
Is there a combinatorial description of those  filter $\F$ on $\N$ for which the sequential completeness of a TVS implies its  $\F$-completeness?
\end{problem}

\begin{problem} \label{prob1+}
Which of concrete filters  $\F$, widely mentioned in literature (like Erd\"os-Ulam filters, summable filters, $f$-statistical filters, filters generated by summability matrices, etc.)  enjoy the property that $\F$-sequential completeness of a TVS implies its  $\F$-completeness?
\end{problem}

Let us consider the following construction. For a free filter $\F$ on $\N$ denote $c(\F)$ the set of all bounded $\F$-convergent numerical sequences. Evidently $c_0 \subset c(\FR) \subset c(\F)  \subset \ell_\infty$. For an ultrafilter $\U$ we have  $c(\F) = \ell_\infty$. Following the argument from Theorem \ref{thm-seq-comp-notultrcomp} one can easily see that the space $(\ell_1, \sigma(\ell_1, c(\F)))$ is not $\F$-complete. So, every time when  $(\ell_1, \sigma(\ell_1, c(\F)))$ is sequentially complete, we obtain an example of a sequentially complete space which is not $\F$-complete. This relates Problem \ref{prob1} the following one.

\begin{problem}  \label{prob2} $ \ $

\begin{enumerate}[(i)]
\item Describe those linear subspaces $E$, $c_0 \subset E  \subset \ell_\infty$, for which the corresponding space  $(\ell_1, \sigma(E)$ is sequentially complete.
\item  Describe those free filters $\F$ on $\N$, for which the corresponding space  $(\ell_1, \sigma(\ell_1, c(\F)))$ is sequentially complete.
\end{enumerate}
\end{problem}

Remark, that for some filters $\F$ the corresponding space  $(\ell_1, \sigma(\ell_1, c(\F)))$ is not sequentially complete. This evidently happens for the the Fr\'{e}chet filter and for those filters $\F$ for which the implication ($(X \in \Compl(\FR)) \Rightarrow (X \in \Compl(\F))$) from Problem \ref{prob1} holds true. Let us give a more advanced example.

\begin{theorem}\label{thm-statist-to-prob2}
For every unbounded modulus function $f$ the space  \newline $(\ell_1, \sigma(\ell_1, c(\F_{f-st})))$  is not sequentially complete.
\end{theorem}
\begin{proof}
Denote $x_n = \frac{1}{n}\sum_{k=1}^n e_k$, where $e_k \in \ell_1$ are the elements of the canonical basis. For every $y = (y_1, y_2, \ldots) \in c(\F_{f-st})$ we know \cite[Corollary 2.2]{aizpuru} that it is statistically convergent to its $f$-statistical limit, so it is Cesaro convergent. Consequently
$$
y(x_n) = \frac{1}{n}\sum_{k=1}^n y_k \xrightarrow[n \to \infty]{} \lim_{\F_{f-st}} (y_n).
$$
This means that $(x_n)$ is a Cauchy sequence in $(\ell_1, \sigma(\ell_1, c(\F_{f-st})))$. On the other hand,  $(x_n)$ is not $\sigma(\ell_1, c(\F_{f-st}))$-convergent to any element of $z \in \ell_1$ because the mapping $y \mapsto  \lim_{\F_{f-st}} (y_n)$ cannot be represented in the form $y \mapsto  \sum_{n \in \N}z_n y_n$.
\end{proof}

\begin{problem}  \label{prob3}
Does there exist a ``universal'' ultrafilter $\U$ on $\N$, such that completeness with respect to $\U$ implies the countable completeness?  Is the existence of such $\U$ consistent with ZFC axioms?
\end{problem}

\begin{problem}   \label{prob4}
For a given TVS $X$ denote $\Compl(X)$ the set of those filters $\F$ for which $X \in \Compl(\F)$. Items (4--6) of Theorem \ref{thm-ComplF1F2} give some restrictions on the structure of $\Compl(X)$. What else can be said about this set? For example, are there any topological  restrictions on the intersection of $\Compl(X)$ with the space $\beta{\N}$ of all ultrafilters?
\end{problem}

Remark that the questions formulated in Problems \ref{prob1}, \ref{prob1+},  \ref{prob3}, and  \ref{prob4} can be asked for smaller classes of spaces, for example for locally convex spaces.

\end{document}